\documentclass[11pt]{amsart}

\usepackage{amsmath,amssymb,amsthm,color,enumerate,comment,bm}
\usepackage{tikz}
\usepackage{url}
\usepackage[colorlinks=true,linkcolor=blue,citecolor=blue,urlcolor=blue]{hyperref}

\newtheorem{theorem}{Theorem}[section]
\newtheorem{proposition}[theorem]{Proposition}
\newtheorem{lemma}[theorem]{Lemma}
\newtheorem{corollary}[theorem]{Corollary}
\theoremstyle{definition}

\newtheorem{remark}[theorem]{Remark}
\newtheorem{question}{Question}

\newcommand{\ZZ}{\mathbb Z}
\newcommand{\RP}{\mathbb{R}P}

\newcommand{\im}{\operatorname{Im}}

\title[Spanning bipartite quadrangulations of projective spaces]
{Spanning bipartite quadrangulations in even triangulations of real projective spaces}
\author[K. Enami]{Kengo Enami}
\address{College of Liberal Arts and Sciences, Kitasato University, 1-15-1 Kitasato, Minami-ku, Sagamihara, Kanagawa 252-0373, Japan}
\email{enamikengo@gmail.com}

\begin{document}

\begin{abstract}
We study spanning quadrangulations of even triangulations
of real projective spaces, using the higher-dimensional
notion of quadrangulation introduced by Kaiser and Stehl\'{\i}k.
We prove that every even triangulation of $\RP^d$, where
$d\ge 2$, admits a spanning bipartite quadrangulation.
Moreover, such a triangulation admits a spanning
non-bipartite quadrangulation if and only if it is not
balanced; if it is balanced, every spanning
quadrangulation is bipartite.
Our proofs use coloring monodromy: its orbits yield
bipartite quadrangulations, while a covering-space
construction gives non-bipartite ones.
More generally, we establish a sufficient condition
for the existence of spanning bipartite quadrangulations
on closed manifolds in terms of the monodromy action,
and derive a consequence involving the fundamental group.
\end{abstract}
\maketitle

\section{Introduction}\label{sec:intro}

\subsection{Spanning bipartite quadrangulations and weak colorings on surfaces}

A \emph{triangulation} of a closed surface $F$ is a graph cellularly embedded in $F$ such that every face is bounded by a triangle.
A \emph{quadrangulation} is defined analogously, with facial walks of length four. A \emph{spanning quadrangulation} of a triangulation $G$ is a spanning subgraph $Q$ whose embedding inherited from $G$ is a quadrangulation of the same surface. Such a subgraph need not be bipartite when the surface is not the sphere, and the question of when it can be chosen bipartite is a natural problem in topological graph theory. It is well known that every triangulation of a closed surface admits a spanning quadrangulation, for instance by deleting the primal edges corresponding to a perfect matching in the cubic dual graph.
However, requiring such a quadrangulation to be
bipartite is substantially more restrictive
\cite{KT,NNO}.

This question is closely related to a coloring problem for hypergraphs.
A \emph{weak $k$-coloring} of a hypergraph is an assignment of at most
$k$ colors to its vertices such that no hyperedge is monochromatic.
The \emph{face hypergraph} of a surface triangulation $G$ has vertex
set $V(G)$ and one hyperedge for the vertex set of each triangular face.
Thus a \emph{weak $2$-coloring of $G$} is a map
$c:V(G)\to\{0,1\}$ such that no triangular face is monochromatic.
Such a coloring need not be proper.
Face-hypergraph colorings of graphs embedded in surfaces were
systematically studied by K\"{u}ndgen and Ramamurthi \cite{KR}.

The following elementary equivalence is fundamental to this paper.

\begin{proposition}\label{prop:surface-equiv}
A triangulation $G$ of a closed surface has a spanning bipartite quadrangulation if and only if $G$ has a weak $2$-coloring.
\end{proposition}
\begin{proof}
Suppose $c:V(G)\to\{0,1\}$ is a weak $2$-coloring.
Let $Q$ be the spanning subgraph of $G$ obtained by deleting
all edges whose endpoints receive the same color.
Since no triangular face of $G$ is monochromatic,
exactly one edge is deleted from each triangular face.
Thus, each pair of triangular faces sharing a deleted edge
forms a quadrangular face of $Q$.
Consequently, $Q$ is a spanning bipartite quadrangulation
of $G$ with bipartition $(c^{-1}(0),c^{-1}(1))$.

Conversely, let $Q$ be a spanning bipartite quadrangulation and color its bipartition classes by $0$ and $1$. Each triangular face of $G$ contains two edges of $Q$, so its vertices cannot all receive the same color. Hence this coloring is weak.
\end{proof}

Since every quadrangulation of the sphere
is bipartite, it follows that every
triangulation of the sphere admits
a spanning bipartite quadrangulation.
On the other hand, such a quadrangulation does not always
exist on other surfaces.
For example, consider $K_6$ embedded as a triangulation
of the projective plane.
Any spanning quadrangulation $Q$ of this triangulation
has $10$ edges by Euler's formula,
whereas a bipartite graph on six vertices has at most
$3\cdot 3=9$ edges.
Hence $K_6$ has no spanning bipartite quadrangulation.

K\"{u}ndgen and Ramamurthi \cite{KR} posed the following question.
The \emph{edge-width} of a graph embedded on a surface
is the length of a shortest noncontractible cycle.

\begin{question}
For every closed surface $S$ other than the sphere,
does there exist a constant $w(S)$ such that
every triangulation of $S$ with edge-width
at least $w(S)$ admits a weak $2$-coloring?
\end{question}

Even if we restrict our attention to triangulations satisfying
a natural parity condition, the problem remains nontrivial.
A triangulation is called \emph{even} if every vertex has even degree.
For even triangulations of the projective plane,
K\"{u}ndgen and Thomassen proved the following result.

\begin{theorem}[K\"{u}ndgen and Thomassen \cite{KT}]\label{thm:KT}
Let $G$ be an even triangulation of the projective plane.
\begin{enumerate}[(i)]
\item If $G$ is $3$-colorable, every spanning quadrangulation
of $G$ is bipartite.
\item If $G$ is not $3$-colorable, then $G$ has both a spanning
bipartite quadrangulation and a spanning non-bipartite
quadrangulation.
\end{enumerate}
\end{theorem}

One of our main results extends this theorem to even
triangulations of the $d$-dimensional real projective space
$\RP^d$, using the notion of quadrangulation introduced by
Kaiser and Stehl\'{\i}k \cite{KS}.

In contrast to the projective plane, an even triangulation
of the torus does not necessarily have a spanning bipartite
quadrangulation.
K\"{u}ndgen and Thomassen \cite{KT} proved that every even
triangulation of the torus has a spanning non-bipartite
quadrangulation and that it also has a spanning bipartite
quadrangulation if its edge-width is sufficiently large.
Nakamoto, Noguchi and Ozeki \cite{NNO} gave a complete
characterization: an even triangulation of the torus has
a spanning bipartite quadrangulation if and only if
it contains no $K_7$ as a subgraph.
In particular, such a quadrangulation exists
if the edge-width of $G$ is at least $4$.
Moreover, they proved that for every closed surface $F$
other than the sphere and the projective plane,
there exists an integer $N(F)$ such that every even
triangulation of $F$ with edge-width at least $N(F)$
has both a spanning bipartite quadrangulation
and a spanning non-bipartite quadrangulation.

Recently, Noguchi \cite{Noguchi2024} completely solved
the existence problem for spanning bipartite quadrangulations
of triangulations of the projective plane, without assuming
evenness, by characterizing all triangulations that admit
such a quadrangulation.

\subsection{Triangulations and quadrangulations in higher dimensions}

We work with finite simplicial complexes,
except when considering covering triangulations,
which may be infinite but are locally finite.
Fix an integer $d\ge 2$ and a connected closed
triangulable $d$-manifold $M$.
Let $T$ be a $d$-dimensional simplicial complex
whose geometric realization $|T|$ is homeomorphic to $M$,
and let $G=T^{(1)}$ be its $1$-skeleton.
We call $G$ a \emph{triangulation of $M$},
with the understanding that the underlying simplicial complex
$T$ is specified.
In particular, the faces and facets of $G$ refer to those of $T$.
We use $T$ when the simplicial structure is relevant
and $G$ when discussing graph-theoretic properties.

A face of $T$ with $k+1$ vertices is called a \emph{$k$-face}.
A \emph{facet} of $T$ is a maximal simplex, so every facet
has $d+1$ vertices.
The \emph{degree} of a $(d-2)$-face $\rho$ is the number
of facets containing $\rho$.
We say that $T$, or its triangulation graph $G$, is
\emph{even} if the degree of every $(d-2)$-face is even.
For $d=2$, these faces are vertices and the definition
reduces to the usual evenness of a surface triangulation.

Since every facet spans a $K_{d+1}$ in $G$, we have
$\chi(G)\ge d+1$.
A triangulation is \emph{balanced} if $\chi(G)=d+1$ (see, e.g., \cite{IKN}).
Every balanced triangulation is even: indeed, fix a
$(d-2)$-face $\rho$.
The facets containing $\rho$ form a cyclic sequence
around it, and the two colors not appearing on $\rho$
must alternate as one goes around the sequence.
Thus the number of such facets is even.
The converse fails on general manifolds, as even
triangulations of $\RP^2$ need not be $3$-colorable.
Nevertheless, for simply connected closed manifolds, such as the
$d$-dimensional sphere $S^d$, evenness and balancedness
coincide, as explained in Section~\ref{sec:monodromy}.

A \emph{weak $k$-coloring} of $T$ is a coloring of its vertex set with at most $k$ colors in which no facet is monochromatic. Its \emph{weak chromatic number} $\chi_{\mathrm w}(T)$ is the least such $k$. This is exactly the weak chromatic number of the $(d+1)$-uniform hypergraph whose hyperedges are the facets of $T$. For a balanced triangulation, merging its $d+1$ colors into two nonempty classes gives a weak $2$-coloring. In particular, every even triangulation of $S^d$ has a weak $2$-coloring, since it is balanced. This contrasts with the case of arbitrary triangulations:
Lee and Nevo \cite{LN} proved that, for every fixed
triangulable $d$-manifold with $d\ge3$, weak chromatic
numbers of its triangulations are unbounded.
In particular, even $S^d$ admits triangulations with
no weak $2$-coloring when $d\ge3$.
This is in sharp contrast to the two-dimensional case,
where every triangulation of $S^2$ admits a weak $2$-coloring.

Kaiser and Stehl\'{\i}k \cite{KS} introduced a higher-dimensional
generalization of quadrangulations on surfaces.
Their motivation was to extend Youngs' theorem \cite{Youngs}
on the chromatic number of quadrangulations
of the projective plane.
(Other notions of higher-dimensional quadrangulations
have also been used to study extensions of Youngs' theorem;
see \cite{EM,HNO,KLNNO}.)
Let $G$ be a triangulation of $M$ with underlying simplicial
complex $T$, not necessarily even.
A \emph{Kaiser--Stehl\'{\i}k quadrangulation} of $G$
is a spanning subgraph $Q$ of $G$ such that, for every facet
$\sigma$ of $T$, the subgraph $Q[V(\sigma)]$ induced by $V(\sigma)$ is a
complete bipartite graph with at least one edge.
Equivalently, $Q[V(\sigma)]$ is isomorphic to $ K_{a,d+1-a}$
for some $1\le a\le d$, where $a$ may depend on $\sigma$.
When $d=2$, the only possibility is $K_{1,2}$,
and this coincides with the usual notion of a spanning
quadrangulation of a surface triangulation.
Throughout this paper, we simply call such a graph
a (spanning) \emph{quadrangulation} of $G$.

The higher-dimensional version of Proposition~\ref{prop:surface-equiv} is equally elementary, but it is worth recording explicitly.

\begin{proposition}\label{prop:weak-equiv}
Let $G$ be a triangulation of a connected
$d$-manifold $M$, where $d\ge 2$.
Then the following are equivalent:
\begin{enumerate}[(i)]
\item $G$ admits a weak $2$-coloring;
\item $G$ has a spanning bipartite quadrangulation.
\end{enumerate}
\end{proposition}

\begin{proof}
Suppose that $G$ admits a weak $2$-coloring $c$.
Let $Q$ be the spanning subgraph of $G$ obtained by
deleting all edges whose endpoints receive the same color.
For each facet $\sigma$ of $G$, there are $a$ vertices
of one color and $d+1-a$ vertices of the other,
where $1\le a\le d$.
Since every facet induces a complete graph in $G$,
we have $Q[V(\sigma)]\cong K_{a,d+1-a}$.
Thus $Q$ is a spanning bipartite quadrangulation of $G$.

Conversely, suppose that $G$ has a spanning bipartite
quadrangulation $Q$.
Color the vertices of $G$ according to the bipartition of $Q$.
Since $Q$ induces a nonempty complete bipartite graph
on every facet of $G$, no facet is monochromatic.
Hence this coloring is a weak $2$-coloring of $G$.
\end{proof}

The following theorem of Kaiser and Stehl\'{\i}k \cite{KS},
which extends Youngs' theorem on non-bipartite
quadrangulations of the projective plane to higher dimensions,
will be used in the proof of our main theorem.

\begin{theorem}[Kaiser and Stehl\'{\i}k \cite{KS}]\label{thm:KS}
If $Q$ is a non-bipartite quadrangulation of $\RP^d$, where $d\ge2$, then $\chi(Q)\ge d+2$.
\end{theorem}

\subsection{Our results}

Our main theorem extends Theorem~\ref{thm:KT}
to even triangulations of real projective spaces
of arbitrary dimension.

\begin{theorem}\label{thm:main}
Let $d\ge2$, and let $G$ be an even triangulation of $\RP^d$.
Then the following hold:
\begin{enumerate}[(i)]
\item If $G$ is balanced, every spanning quadrangulation
of $G$ is bipartite.
\item If $G$ is not balanced, then $G$ has both
a spanning bipartite quadrangulation and
a spanning non-bipartite quadrangulation.
\end{enumerate}
\end{theorem}

Theorem~\ref{thm:main} recovers Theorem~\ref{thm:KT}
when $d=2$ and yields the following corollary
for every $d\ge2$.

\begin{corollary}\label{cor:mainexist}
Every even triangulation of $\RP^d$, for $d\ge2$,
has a spanning bipartite quadrangulation.
\end{corollary}

Our proof uses \emph{coloring monodromy}, a homomorphism
$\mu_G:\pi_1(M)\to S_{d+1}$ associated with an even
triangulation $G$ of a connected closed $d$-manifold $M$.
We introduce this notion in Section~\ref{sec:monodromy}.
For projective spaces, the orbits of the monodromy
action induce a partition of $V(G)$ that yields
a bipartite quadrangulation, while a covering-space
argument is used for the non-bipartite case.

After proving our main theorem in Section~3,
we present further observations on quadrangulations
of projective spaces and discuss extensions to
general closed manifolds in Section~\ref{sec:general}.
In particular, we show that every even triangulation
of $\RP^3$ admits a spanning bipartite quadrangulation
that induces $K_{2,2}$ on each tetrahedron,
giving rise to a quadrangulated Heegaard surface.
We then extend our approach beyond projective spaces
to more general closed manifolds.
In particular, we obtain the following sufficient
condition for the existence of a spanning
bipartite quadrangulation.

\begin{theorem}\label{thm:general}
Let $G$ be an even triangulation of a connected
closed $d$-manifold $M$, where $d\ge2$.
If the image of $\mu_G$ does not act transitively
on $[d+1]$, then $G$ has
a spanning bipartite quadrangulation.
\end{theorem}

As a consequence, we obtain a sufficient condition
depending only on the fundamental group of $M$.

\begin{corollary}\label{cor:index}
Let $M$ be a connected closed triangulable
$d$-manifold, where $d\ge2$.
If $\pi_1(M)$ has no subgroup of index $d+1$,
then every even triangulation of $M$
has a spanning bipartite quadrangulation.
\end{corollary}


\section{Coloring monodromy}\label{sec:monodromy}

In this section, we construct the monodromy representation explicitly from facet walks and permutations.
The key topological structure is the dual cell structure of a triangulated manifold.
Coloring monodromy and related concepts have been studied in both higher-dimensional settings \cite{Fisk,IKN,Joswig,RT} and surface triangulations \cite{HNOS,HRS,KNS,Mohar,NNOextension,Noguchi2026}.
We recall the construction because its precise form is central to our proofs.

Let $G$ be a triangulation of a connected closed
$d$-manifold $M$, with underlying simplicial complex $T$.
Two facets $\sigma$ and $\sigma'$ of $T$ are
\emph{adjacent} if they share a $(d-1)$-face.
Let $x$ and $x'$ be the unique vertices in
$V(\sigma)\setminus V(\sigma')$ and
$V(\sigma')\setminus V(\sigma)$, respectively.
There is a canonical bijection
$\varphi_{\sigma\sigma'}:V(\sigma)\longrightarrow V(\sigma')$
that fixes all vertices in $\sigma\cap\sigma'$ and maps
$x$ to $x'$.
Thus, if the vertices of $\sigma$ are colored with
the $d+1$ distinct colors in
$[d+1]:=\{1,\ldots,d+1\}$,
there is a unique coloring of $\sigma'$ that agrees
with the original coloring on their common face.
We call this operation \emph{color transport}
across their common $(d-1)$-face.

The \emph{dual graph} $D(T)$ has the facets of $T$
as vertices, with an edge for each pair of adjacent facets.
A \emph{facet walk} is a sequence
$W=(\sigma_0,\ldots,\sigma_m)$
in which consecutive facets are adjacent.
Composing the above bijections gives a bijection
$\varphi_W:V(\sigma_0)\to V(\sigma_m)$.
Fixing a facet $\sigma_*$ and an initial bijective coloring
$c_*:V(\sigma_*)\to[d+1]$,
a closed facet walk $W$ based at $\sigma_*$
determines a permutation of $[d+1]$
by comparing the initial and final color assignments.
More explicitly, the coloring transported back to
$\sigma_*$ is $c_*\circ\varphi_W^{-1}$,
and we associate to $W$ the permutation
$c_*\circ\varphi_W^{-1}\circ c_*^{-1}$.
With the convention that the product of two walks
traverses the first walk and then the second,
this assignment respects composition.

For a face $\rho$ of $T$, its \emph{link}
$\operatorname{lk}_T(\rho)$ is the simplicial complex
consisting of all faces $\eta$ of $T$ such that
$\eta\cap\rho=\emptyset$ and $\rho\cup\eta$
is a face of $T$.
Since $M$ is a closed manifold,
the link of every $(d-2)$-face is a cycle.

Inserting or deleting a backtrack
$(\sigma,\sigma',\sigma)$
in a facet walk does not change the induced permutation.
The effect of going once around a $(d-2)$-face
is described by the following lemma.

\begin{lemma}\label{lem:localparity}
Let $G$ be a triangulation of a connected closed
$d$-manifold $M$, and let $\rho$ be a $(d-2)$-face
of its underlying simplicial complex $T$.
Let $m$ be the degree of $\rho$.
Transporting a local $(d+1)$-coloring once around
the facets containing $\rho$ induces the identity
permutation when $m$ is even and a transposition
when $m$ is odd.
In the latter case, the transposition interchanges
the two colors absent from $\rho$.
\end{lemma}

\begin{proof}
The link of $\rho$ is a cycle of length $m$,
whose edges correspond to the facets containing $\rho$.
The $d-1$ vertices of $\rho$ keep their colors
during transport, while the two remaining colors
alternate as one moves from one incident facet
to the next.
After $m$ steps, these two colors return to their
original positions if $m$ is even, and are exchanged
if $m$ is odd.
\end{proof}

We now assume that $G$ is even.
We explain why the permutations associated with closed
facet walks depend only on their homotopy classes.
The dual graph $D(T)$ can be regarded as the $1$-skeleton
of the cell decomposition dual to $T$,
which can be constructed using the barycentric subdivision.
For each $(d-2)$-face $\rho$ of $T$,
the corresponding dual $2$-cell has as its boundary
the closed facet walk around $\rho$.
The dual $2$-skeleton has the same fundamental group
as $M$, since attaching cells of dimension at least three
does not change the fundamental group.

Consequently, homotopies of closed facet walks
are generated by inserting or deleting backtracks
and by moving across dual $2$-cells.
As observed above, backtracks do not affect color transport.
Moreover, since $G$ is even,
Lemma~\ref{lem:localparity} shows that transporting colors
around the boundary of any dual $2$-cell
induces the identity permutation.
Thus, the permutation associated with a closed facet walk
depends only on its homotopy class.

It follows that color transport induces a well-defined
homomorphism
\[
\mu_G:\pi_1(M)\longrightarrow S_{d+1},
\]
called the \emph{coloring monodromy} of $G$.
Here the base point of $\pi_1(M)$ is chosen in the
interior of the initial facet $\sigma_*$.
Changing the initial coloring $c_*$ conjugates
$\mu_G$ by an element of $S_{d+1}$.
Changing the initial facet $\sigma_*$ also gives
a conjugate representation, after identifying
the fundamental groups along a path.
Consequently, the triviality of $\mu_G$ and
the orbit structure of its image are independent
of these choices.

The following equivalence is well known;
see, for example, \cite{Joswig,IKN}.

\begin{proposition}\label{prop:balanced}
Let $G$ be an even triangulation of a connected
closed $d$-manifold $M$.
Then $G$ is balanced if and only if
$\mu_G$ is trivial.
\end{proposition}

\begin{proof}
If $G$ is balanced, a proper $(d+1)$-coloring
is preserved under color transport along any
facet walk. Hence $\mu_G$ is trivial.

Conversely, suppose that $\mu_G$ is trivial.
Fix a facet $\sigma_*$ and a bijective coloring
of its vertices.
Since the dual graph is connected, this coloring
can be transported to every facet.
The triviality of $\mu_G$ ensures that the resulting
coloring is independent of the chosen facet walk.
These local colorings therefore define a proper
$(d+1)$-coloring of $G$, showing that $G$ is balanced.
\end{proof}

\begin{corollary}\label{cor:simplyconn}
An even triangulation of a simply connected closed $d$-manifold is balanced. In particular, a triangulation of $S^d$ is even if and only if it is balanced.
\end{corollary}
\begin{proof}
If $M$ is simply connected, then $\pi_1(M)$ is trivial, and hence so is $\mu_G$. Apply Proposition~\ref{prop:balanced}, and recall that balancedness always implies evenness.
\end{proof}

Let $\Gamma=\im\mu_G\le S_{d+1}$.
For each $i\in[d+1]$, the \emph{orbit} of $i$
under $\Gamma$ is defined by
\[
O(i)=\{\gamma(i):\gamma\in\Gamma\}.
\]
The orbits of $\Gamma$ partition $[d+1]$.
As we show below, this partition induces
a partition of the vertices of $G$.

\begin{lemma}\label{lem:orbitpartition}
Let $G$ be an even triangulation of a connected closed
$d$-manifold $M$, and let $T$ be its underlying
simplicial complex.
If $O_1,\ldots,O_k$ are the orbits of
$\im\mu_G$ on $[d+1]$, then there is a partition
\[
V(G)=V_1\sqcup\cdots\sqcup V_k
\]
such that
\[
|V_i\cap V(\sigma)|=|O_i|
\]
for every facet $\sigma$ of $T$ and every $i\in[k]$.
\end{lemma}

\begin{proof}
Fix an initial coloring $c_*$ of a facet $\sigma_*$.
For each facet $\sigma$, transport $c_*$ along a
facet walk from $\sigma_*$ to $\sigma$.
Different choices of the walk give colorings that
differ by a permutation in $\im\mu_G$.
Thus, although the color of a vertex may depend on
the walk, the orbit containing that color does not.
Moreover, transport across adjacent facets preserves
the colors of their common vertices.
Hence the orbit labels agree on overlaps and define
a partition $V(G)=V_1\sqcup\cdots\sqcup V_k$.
Since each transported coloring uses all $d+1$
colors exactly once on every facet, the asserted
cardinalities follow.
\end{proof}

In particular, any nonempty proper union of the sets
$V_1,\ldots,V_k$ defines a weak $2$-coloring of $G$.
We will use this observation both for projective spaces
and for general closed manifolds.

\section{Proof of the main theorem}\label{sec:proof}

Let $G$ be an even triangulation of $\RP^d$,
where $d\ge2$, with underlying simplicial complex $T$.
We first use the orbit partition from
Lemma~\ref{lem:orbitpartition} to obtain a bipartite
quadrangulation. The non-bipartite construction
will require an additional covering-space argument.

\subsection{Existence of a spanning bipartite quadrangulation}

Since $\pi_1(\RP^d)\cong\ZZ_2$, the image of
$\mu_G$ is generated by a permutation
$\tau\in S_{d+1}$ satisfying $\tau^2=1$.
Every orbit of $\tau$ has size at most two.
Since $d+1\ge3$, there must be at least two
orbits, and hence $\im\mu_G$ does not act
transitively on $[d+1]$.

Let $O_1,\ldots,O_k$ be these orbits.
By Lemma~\ref{lem:orbitpartition}, they induce
a partition $V(G)=V_1\sqcup\cdots\sqcup V_k$
such that each facet contains exactly $|O_i|$
vertices from $V_i$.
Since $k\ge2$, assigning color $0$ to $V_1$
and color $1$ to all remaining vertices gives
a weak $2$-coloring of $G$.
Proposition~\ref{prop:weak-equiv} therefore yields
a spanning bipartite quadrangulation.
This proves Corollary~\ref{cor:mainexist}.

\subsection{The balanced case}

Suppose that $G$ is balanced.
Then $\chi(G)=d+1$, and hence
$\chi(Q)\le d+1$ for every spanning subgraph
$Q$ of $G$.
If $Q$ were a non-bipartite quadrangulation
of $G$, Theorem~\ref{thm:KS} would imply
$\chi(Q)\ge d+2$, a contradiction.
Thus every spanning quadrangulation of $G$
is bipartite.
Since the existence of a spanning bipartite
quadrangulation was established above,
this proves part~(i) of Theorem~\ref{thm:main}.

\subsection{Constructing a non-bipartite quadrangulation in the non-balanced case}

We first record two elementary facts about
bichromatic-edge graphs and graph coverings.

\begin{lemma}\label{lem:edgeconstruction}
Let $G$ be a triangulation of a connected
$d$-manifold, where $d\ge2$, and let
$c:V(G)\to\{0,1\}$ be a weak $2$-coloring.
Then the spanning subgraph $Q_c$ obtained by deleting
all edges whose endpoints receive the same color
is a connected bipartite quadrangulation of $G$.
\end{lemma}

\begin{proof}
By Proposition~\ref{prop:weak-equiv},
$Q_c$ is a spanning bipartite quadrangulation of $G$.
For every facet $\sigma$ of the underlying
simplicial complex $T$, the graph $Q_c[V(\sigma)]$
is connected.
Since the dual graph $D(T)$ is connected
and adjacent facets share vertices,
it follows that $Q_c$ is connected.
\end{proof}

A \emph{deck transformation} of a covering map
$p:\widetilde X\to X$ is an automorphism
$\gamma$ of $\widetilde X$ satisfying
$p\circ\gamma=p$.
For the universal cover of a manifold $M$,
the deck transformation group can be identified
with $\pi_1(M)$.

\begin{lemma}\label{lem:oddquotient}
Let $p:H\to Q$ be a graph covering,
where $H$ is connected and bipartite.
If some deck transformation of $p$ interchanges
the two partite sets of $H$, then $Q$ is not bipartite.
\end{lemma}

\begin{proof}
Let $\gamma$ be such a deck transformation
and choose a vertex $v$ of $H$.
Since $H$ is connected, there is a path $P$
from $v$ to $\gamma v$.
As these vertices belong to different partite sets,
$P$ has odd length.
Since $p(v)=p(\gamma v)$,
the image of $P$ under $p$ is a closed walk
of odd length in $Q$.
Therefore, $Q$ is not bipartite.
\end{proof}

We now relate coloring monodromy to the universal
cover of the triangulated manifold.
This relation will also be used in
Section~\ref{sec:general}.

\begin{proposition}\label{prop:cover}
Let $G$ be an even triangulation of a connected
closed $d$-manifold $M$, and let
$p:\widetilde M\to M$ be its universal covering map.
Write $\widetilde T$ for the lifted simplicial complex
and $\widetilde G=\widetilde T^{(1)}$.
Then $\widetilde G$ admits a proper $(d+1)$-coloring
$\widetilde c:V(\widetilde G)\to[d+1]$.
After suitable choices of the initial coloring and the
identification of deck transformations with $\pi_1(M)$,
we have
\[
\widetilde c(\gamma v)
=\mu_G(\gamma)\bigl(\widetilde c(v)\bigr)
\qquad
(\gamma\in\pi_1(M),\ v\in V(\widetilde G)).
\]
\end{proposition}

\begin{proof}
The lifted triangulation is even, and
$\widetilde M$ is simply connected.
By the same argument as in
Section~\ref{sec:monodromy},
color transport from an initial facet is
independent of the chosen facet walk.
It therefore defines a proper
$(d+1)$-coloring $\widetilde c$
of $\widetilde G$.

For each deck transformation $\gamma$,
the coloring
$v\mapsto\widetilde c(\gamma v)$
is also a proper $(d+1)$-coloring
of $\widetilde G$.
Any two proper $(d+1)$-colorings of a connected
triangulation differ by a permutation of the colors.
Indeed, such a permutation is uniquely determined
on one facet and extends consistently across
adjacent facets.
Thus there is a unique $s_\gamma\in S_{d+1}$
such that
\[
\widetilde c(\gamma v)
=s_\gamma(\widetilde c(v))
\]
for every $v\in V(\widetilde G)$.

It remains to show that $s_\gamma=\mu_G(\gamma)$.
Choose a lift $\widetilde\sigma_*$ of the initial
facet $\sigma_*$ used to define $\mu_G$,
and let $c_*$ be the coloring induced by
$\widetilde c$ on $\sigma_*$.
Let $W$ be a closed facet walk based at
$\sigma_*$ representing $\gamma$.
Under the chosen identification of
$\pi_1(M)$ with the deck transformation group,
the lift of $W$ starting at
$\widetilde\sigma_*$ ends at
$\gamma\widetilde\sigma_*$.

Since $\widetilde c$ is a proper coloring,
color transport along the lifted walk
agrees with $\widetilde c$ on every facet.
Identifying the initial and terminal facets
with $\sigma_*$ via the covering map,
their colorings are $c_*$ and
$s_\gamma\circ c_*$, respectively.
Thus the permutation induced by color transport
along $W$ is precisely $s_\gamma$.
By the definition of coloring monodromy,
we obtain $\mu_G(\gamma)=s_\gamma$.
\end{proof}

Now suppose that $G$ is not balanced.
Let $p:S^d\to\RP^d$ be the antipodal double cover,
let $\widetilde G$ and $\widetilde T$ be the lifted
triangulation graph and simplicial complex,
and let $\alpha$ be the nontrivial deck transformation.
By Proposition~\ref{prop:cover}, choose a proper
$(d+1)$-coloring $\widetilde c$ of $\widetilde G$
such that $\widetilde c(\alpha v)=\tau(\widetilde c(v))$,
where $\tau=\mu_G(\gamma)$ and $\gamma$ is the
nonidentity element of $\pi_1(\RP^d)\cong\ZZ_2$.
By Proposition~\ref{prop:balanced}, $\tau$ is nontrivial.
Since $\tau^2=1$, it exchanges at least one pair of colors.
After relabeling the colors if necessary, we may assume that
$\tau(1)=2$ and $\tau(2)=1$.

We now construct a $2$-coloring
$\varphi:V(\widetilde G)\to\{0,1\}$
such that
\begin{equation}\label{eq:antisym}
\varphi(\alpha v)=1-\varphi(v)
\qquad(v\in V(\widetilde G)).
\end{equation}

First, assign color $0$ to the vertices in $\widetilde c^{-1}(1)$ and color $1$ to those in $\widetilde c^{-1}(2)$.
For each other $2$-element orbit $\{i,j\}$ of $\tau$,
assign value $0$ to all vertices of one color class
and value $1$ to all vertices of the other.

It remains to consider the colors fixed by $\tau$.
Suppose that $\tau(i)=i$.
Then $\alpha$ preserves the color class
$\widetilde c^{-1}(i)$.
Since $\alpha$ has no fixed vertices,
this color class is partitioned into pairs
$\{v,\alpha v\}$.
For each pair, assign value $0$ to one vertex
and value $1$ to the other.
These choices can be made arbitrarily.
By construction, $\varphi$ satisfies
\eqref{eq:antisym}.

We claim that $\varphi$ is a weak $2$-coloring.
Indeed, every facet of $\widetilde T$ contains
exactly one vertex of $\widetilde c^{-1}(1)$
and exactly one vertex of $\widetilde c^{-1}(2)$.
These two vertices receive different values under
$\varphi$, so no facet is monochromatic.

Let $\widetilde Q$ be the spanning subgraph of
$\widetilde G$ consisting of all edges whose
endpoints receive different values under $\varphi$.
By Lemma~\ref{lem:edgeconstruction},
$\widetilde Q$ is a connected bipartite
quadrangulation of $\widetilde G$.
By \eqref{eq:antisym}, an edge of $\widetilde G$
has differently colored endpoints if and only if
its image under $\alpha$ does.
Thus $\alpha$ preserves $\widetilde Q$,
and $\widetilde Q$ induces a spanning
subgraph $Q$ of $G$.
The restriction $p:\widetilde Q\to Q$
is a graph covering.

We show that $Q$ is a quadrangulation of $G$.
Let $\sigma$ be any facet of $T$, and choose
a lift $\widetilde\sigma$ in $\widetilde T$.
Since $\varphi$ is a weak $2$-coloring,
$\widetilde Q[V(\widetilde\sigma)]$
is a nonempty complete bipartite graph
on all vertices of $\widetilde\sigma$.
The covering map sends $\widetilde\sigma$
bijectively onto $\sigma$.
Moreover, an edge of $\sigma$ belongs to $Q$
if and only if its lift in $\widetilde\sigma$
belongs to $\widetilde Q$.
Thus $Q[V(\sigma)]$ is isomorphic to
$\widetilde Q[V(\widetilde\sigma)]$,
and hence is a nonempty complete bipartite graph.
Therefore $Q$ is a spanning quadrangulation of $G$.

Finally, \eqref{eq:antisym} shows that $\alpha$
interchanges the two bipartition classes of
$\widetilde Q$.
Since $\widetilde Q$ is connected,
Lemma~\ref{lem:oddquotient} implies that
$Q$ is not bipartite.
This proves part~(ii) of Theorem~\ref{thm:main}
and completes the proof.\qed

\section{Further consequences and generalizations}\label{sec:general}

We begin with two observations on spanning bipartite
quadrangulations.
The first concerns balanced triangulations of
general closed manifolds, while the second gives
a stronger conclusion for even triangulations
of $\RP^3$.

\begin{proposition}\label{prop:projective-types}
Let $G$ be a triangulation of a connected closed
$d$-manifold $M$, where $d\ge2$, and let $T$
be its underlying simplicial complex.
\begin{enumerate}
\item[(i)]
If $G$ is balanced, then for every integer
$1\le a\le d$, there exists a spanning bipartite
quadrangulation $Q$ of $G$ such that
$Q[V(\sigma)]$ is isomorphic to $ K_{a,d+1-a}$
for every facet $\sigma$ of $T$.

\item[(ii)]
If $M=\RP^3$ and $G$ is even, then there exists
a spanning bipartite quadrangulation $Q$ of $G$
such that
$Q[V(\sigma)]$ is isomorphic to $K_{2,2}$
for every facet $\sigma$ of $T$,
regardless of whether $G$ is balanced.
\end{enumerate}
\end{proposition}

\begin{proof}
Suppose first that $G$ is balanced.
Let $c:V(G)\to[d+1]$ be a proper coloring.
For any $1\le a\le d$, partition the color set
into $A=\{1,\ldots,a\}$ and
$B=[d+1]\setminus A$.
Assign color $0$ to the vertices in $c^{-1}(A)$
and color $1$ to those in $c^{-1}(B)$.
Since every facet of $T$ contains exactly one
vertex of each color under $c$, this is a weak
$2$-coloring with exactly $a$ vertices of color $0$
and $d+1-a$ vertices of color $1$ in each facet.
By Proposition~\ref{prop:weak-equiv}, it induces
a spanning bipartite quadrangulation
with the required property.

Now suppose that $M=\RP^3$ and $G$ is even.
Since $\pi_1(\RP^3)\cong\ZZ_2$,
the image of $\mu_G$ is generated by a permutation
$\tau\in S_4$ satisfying $\tau^2=1$.
After relabeling the colors, we may assume that
$\tau$ is one of $1$, $(12)$ and $(12)(34)$.
In each case, $A=\{1,2\}$ is invariant under $\tau$.
As $A$ is a union of monodromy orbits,
Lemma~\ref{lem:orbitpartition} gives a partition
$V(G)=V_0\sqcup V_1$ such that every facet of $T$
contains exactly two vertices from each part.
Coloring $V_0$ and $V_1$ differently gives a weak
$2$-coloring, and Proposition~\ref{prop:weak-equiv}
yields the required quadrangulation.
\end{proof}

\begin{remark}
The quadrangulation in
Proposition~\ref{prop:projective-types}(ii)
also gives rise to a Heegaard splitting of
$\RP^3$.

Let $V_0$ and $V_1$ be the two color classes
of $Q$.
Each tetrahedron of $T$ contains exactly two
vertices from each class, and its four
bichromatic edges form a copy of $K_{2,2}$.

Define a piecewise linear map
$f:\RP^3\to[0,1]$ by assigning $0$ to
the vertices in $V_0$ and $1$ to those in $V_1$,
and extending linearly over each tetrahedron.
Then $\Sigma=f^{-1}(1/2)$ intersects each
tetrahedron in a quadrilateral disk whose
vertices are the midpoints of its four
bichromatic edges.
These disks fit together to form a closed
embedded surface $\Sigma$, which admits
a quadrangulation with one quadrilateral face
for each tetrahedron of $T$.

Moreover, the induced subgraphs
$G[V_0]$ and $G[V_1]$ are both connected.
Indeed, each tetrahedron contains a
monochromatic edge of each color, and
the corresponding edges in adjacent tetrahedra
share a vertex of that color.
Since the dual graph of $T$ is connected,
the assertion follows.

The two submanifolds $f^{-1}([0,1/2])$
and $f^{-1}([1/2,1])$ are regular neighborhoods
of $G[V_0]$ and $G[V_1]$, respectively.
Since both graphs are connected, these
neighborhoods are handlebodies.
Consequently, $\Sigma$ is a Heegaard surface
of $\RP^3$.
This is an instance of the construction
of Heegaard splittings from vertex partitions
described by Rubinstein and Tillmann
\cite{RubinsteinTillmann2018}.
\end{remark}

The preceding observations extend naturally to
even triangulations of arbitrary connected closed
manifolds.
The key point is that a nontrivial partition
into monodromy orbits induces a weak $2$-coloring
of the triangulation itself.
We first prove Theorem~\ref{thm:general}.

\begin{proof}[Proof of Theorem~\ref{thm:general}]
Since $\im\mu_G$ acts intransitively on $[d+1]$,
it has at least two orbits.
By Lemma~\ref{lem:orbitpartition}, these orbits
induce a partition of $V(G)$ into at least two
nonempty sets, each meeting every facet of $G$.
Assigning color $0$ to one of these sets and
color $1$ to all remaining vertices gives a weak
$2$-coloring of $G$.
Proposition~\ref{prop:weak-equiv} then yields
a spanning bipartite quadrangulation.
\end{proof}

For comparison, we give a sufficient condition
for the existence of a non-bipartite quadrangulation.
The idea is to use two sets of colors that may be
interchanged by the monodromy, while the remaining
colors form an invariant set.

\begin{proposition}\label{prop:swappedpartition}
Let $G$ be an even triangulation of a connected
closed $d$-manifold $M$, where $d\ge2$.
Suppose that $[d+1]=A\sqcup B\sqcup C$,
where $|A|=|B|\ge1$ and $C$ may be empty.
Assume that every permutation in $\im\mu_G$
either preserves $A$ and $B$ individually
or interchanges them, and that at least one
permutation interchanges them.
Then $G$ admits a spanning non-bipartite
quadrangulation.
\end{proposition}

\begin{proof}
Let $\widetilde c:V(\widetilde G)\to[d+1]$
be a proper coloring of the universal covering
triangulation, as in Proposition~\ref{prop:cover}.

We construct a weak $2$-coloring
$\varphi:V(\widetilde G)\to\{0,1\}$.
First, assign value $0$ to the vertices in
$\widetilde c^{-1}(A)$ and value $1$ to those
in $\widetilde c^{-1}(B)$.
It remains to color the vertices in
$\widetilde c^{-1}(C)$.
Since every permutation in $\im\mu_G$
preserves $C$, this set is invariant under
all deck transformations.
For each vertex $v$ of $G$ whose lifts belong
to $\widetilde c^{-1}(C)$, choose one lift
$\widetilde v_0$.
Every lift of $v$ is uniquely of the form
$\gamma\widetilde v_0$ for some
$\gamma\in\pi_1(M)$.
Assign $\varphi(\gamma\widetilde v_0)=0$
if $\mu_G(\gamma)$ preserves $A$ and $B$
individually, and $1$ if it interchanges them.
By construction, every deck transformation
either preserves both color classes of $\varphi$
or interchanges them, according to its
monodromy action on $A$ and $B$.
Moreover, every facet of $\widetilde T$
contains vertices whose colors under
$\widetilde c$ belong to $A$ and $B$,
respectively.
Hence $\varphi$ is a weak $2$-coloring.

Let $\widetilde Q$ be the spanning subgraph
of $\widetilde G$ consisting of all edges
whose endpoints receive different values
under $\varphi$.
By Lemma~\ref{lem:edgeconstruction},
$\widetilde Q$ is a connected bipartite
quadrangulation of $\widetilde G$.
Since every deck transformation preserves
$\widetilde Q$, it descends to a spanning
subgraph $Q$ of $G$.
The restriction of the covering map to
$\widetilde Q$ is a graph covering onto $Q$.

We show that $Q$ is a quadrangulation of $G$.
Let $\sigma$ be any facet of $T$, and choose
a lift $\widetilde\sigma$ in $\widetilde T$.
Since $\varphi$ is a weak $2$-coloring,
$\widetilde Q[V(\widetilde\sigma)]$
is a nonempty complete bipartite graph.
The covering map sends $\widetilde\sigma$
bijectively onto $\sigma$ and preserves
the selected edges.
Hence $Q[V(\sigma)]$ is also a nonempty
complete bipartite graph.
Therefore $Q$ is a spanning quadrangulation
of $G$.

Finally, by assumption, some deck transformation
interchanges the two bipartition classes
of $\widetilde Q$.
Since $\widetilde Q$ is connected,
Lemma~\ref{lem:oddquotient} implies that
$Q$ is non-bipartite.
\end{proof}

We next prove the fundamental-group criterion
stated in Corollary~\ref{cor:index}.

\begin{proof}[Proof of Corollary~\ref{cor:index}]
Let $G$ be an arbitrary even triangulation of $M$.
Suppose that $\im\mu_G$ acts transitively on $[d+1]$.
Fix a color $i\in[d+1]$, and let $H$ be
the subgroup of $\pi_1(M)$ consisting of
elements whose monodromy fixes $i$.

For each $j\in[d+1]$, choose an element
$g_j\in\pi_1(M)$ whose monodromy sends $i$ to $j$.
Then $g_jH$ consists precisely of the elements
whose monodromy sends $i$ to $j$.
Thus these $d+1$ cosets partition $\pi_1(M)$,
so $H$ has index $d+1$.
This contradicts the assumption that
$\pi_1(M)$ has no subgroup of index $d+1$.
Hence $\im\mu_G$ is not transitive,
and Theorem~\ref{thm:general} applies.
\end{proof}

For instance, if $\pi_1(M)$ is finite
and its order is not divisible by $d+1$,
then the hypothesis of
Corollary~\ref{cor:index} holds.
In particular, for a three-dimensional
lens space $L(p,q)$, whose fundamental group
is cyclic of order $p$, the corollary applies
whenever $4$ does not divide $p$.
This includes $\RP^3\cong L(2,1)$.

Our results  illustrate the usefulness of
coloring monodromy in constructing spanning
quadrangulations of even triangulations,
both in real projective spaces and in
more general closed manifolds.

\section*{Acknowledgements}
This work was supported by JSPS KAKENHI Grant Number JP23K13006.

\section*{Data availability}
No datasets were generated or analyzed during this study.

\section*{Declaration of generative AI and AI-assisted technologies in the manuscript preparation process}
The author used ChatGPT to improve the manuscript's language, structure, and readability, and utilized ChatGPT to verify mathematical consistency. The author assume full responsibility for the content, having reviewed and edited all AI-assisted output.

\end{document}